\documentclass[runningheads]{llncs}
\usepackage{url}
\usepackage{multirow}
\usepackage{amssymb}
\usepackage{amsxtra}
\usepackage{makeidx}
\usepackage{amsfonts}
\usepackage{mathtools}
\usepackage{algorithm}
\usepackage{algorithmic}
\usepackage{float}
\usepackage{graphicx}
\usepackage{comment}
\usepackage[font=small,labelfont=bf]{caption}
\usepackage{subcaption}
\usepackage{breqn}
\usepackage[english]{babel}

\usepackage{color}
\usepackage{hyperref}

\DeclareMathOperator*{\argmin}{arg\,min}

\begin{document}
\title{Algorithms for Solving Variational inequalities and Saddle Point Problems with some Generalizations of Lipschitz Property for Operators\thanks{The research in Introduction and Sections 2, 3 is supported by the Ministry of Science and Higher Education of the Russian Federation (Goszadaniye in MIPT, project 075-00337-20-03).
The research in Algorithm 3 was partially supported by the grant of the President of Russian Federation for young candidates of sciences (project MK-15.2020.1). The research in Theorem 1 and partially in section 2.1 was supported by the Russian Science Foundation (project 18-71-10044).}}
\titlerunning{Algorithms for Variational Inequalities and Saddle Point Problems}

\author{Alexander A. Titov \inst{1,3} \orcidID{0000-0001-9672-0616} \and Fedor S. Stonyakin  \inst{1,}\inst{2} \orcidID{0000-0002-9250-4438}\and Mohammad S. Alkousa\inst{1,3} \orcidID{0000-0001-5470-0182} \and Alexander V. Gasnikov \inst{1,3,4,5} \orcidID{0000-0002-5982-8983}}
	
\authorrunning{A. Titov et al.}

\institute{Moscow Institute of Physics and Technology, Moscow, Russia
	\and
	V.\,I.\,Vernadsky Crimean Federal University, Simferopol, Russia\\
	\and 
	HSE University, Moscow, Russia
	\and
	Institute for Information Transmission Problems RAS, Moscow, Russia
	\and
	Caucasus Mathematical Center, Adyghe State University, Maikop, Russia\\
	\email{a.a.titov@phystech.edu, fedyor@mail.ru,
		mohammad.alkousa@phystech.edu,
		gasnikov@yandex.ru}
	}

\maketitle
\begin{abstract}
The article is devoted to the development of numerical methods for solving saddle point problems and variational inequalities with simplified requirements for the smoothness conditions of functionals. Recently, some notable methods for optimization problems with strongly monotone operators were proposed. Our focus here is on newly proposed techniques for solving strongly convex-concave saddle point problems. One of the goals of the article is to improve the obtained estimates of the complexity of introduced algorithms by using accelerated methods for solving auxiliary problems. The second focus of the article is introducing an analogue of the boundedness condition for the operator in the case of arbitrary (not necessarily Euclidean) prox structure. We propose an analogue of the Mirror Descent method for solving variational inequalities with such operators, which is optimal in the considered class of problems.

\end{abstract}

\section*{Introduction}\label{sec1_introduction}
Modern numerical optimization methods are widely used in solving problems in various fields of science. The problem of finding the optimal value in the investigated mathematical model naturally arises in machine learning, data analysis, economics, market equilibrium problems, electric power systems, control theory, optimal transport, molecular modeling, etc. This paper is devoted to the development and analysis of numerical methods for solving saddle point problems and variational inequalities. 
Both saddle point problems and variational inequalities play a critical role in structural analysis, resistive networks, image processing, zero-sum game, and Nash equilibrium problems \cite{Saddle,Jud2011,Nemir2004,Scutari}. 

Despite variational inequalities and saddle point problems are closely related (it will be discussed later), the article consists of two conditionally independent parts, each devoted to the special type of generalization of smoothness conditions and the corresponding problem. 

Firstly, the paper focuses on recently proposed numerical methods for solving ($\mu_x,\mu_y$)-strongly convex-concave saddle point problems of the following form:
\begin{equation}\label{main}
\min\limits_{x}\max\limits_{y} f(x,y).
\end{equation} 

 In particular, in \cite{ST} authors considered a modification of some well-known scheme for speeding up methods for solving {\it smooth} saddle point problems. 
The smoothness, in such a setting, means the Lipschitz continuity of all partial gradients of the objective function $f$. We improve the obtained estimates of the complexity for the case of a non-smooth objective. 

In more details, we consider the saddle point problem under the assumption, that one of the partial gradients still satisfies the Lipschitz condition, while the other three satisfy simplified smoothness condition, namely, the H\"older continuity. Note, it is the Lipschitz continuity, which makes it possible to use accelerated methods for solving auxiliary problems. The H\"older continuity is an important generalization of the Lipschitz condition, it appears in a large number of applications \cite{NesterovUGM,UMP}. In particular, if a function is uniformly convex, then its conjugated will necessarily have the H\"older-continuous gradient \cite{NesterovUGM}.

We prove, that the strongly convex-concave saddle point problem, modified in a functional form, admits an inexact $(\delta,L,\mu_x)$--model, where $\delta$ is comparable to $\varepsilon$, and apply the Fast Gradient Method to achieve the $\varepsilon$--solution of the considered problem. The total number of iterations is estimated as follows:
$$
    \mathcal{O}\left(\sqrt{\frac{L}{\mu_x}}\cdot\sqrt{\frac{L_{yy}}{\mu_y}}\cdot \log\frac{2L_{yy}R^2}{\varepsilon}\cdot\log{\frac{2LD^2}{\varepsilon}}\right),
$$
where for $0 \leq \nu < 1$ $L = \tilde{L}\left( \frac{\tilde{L} (1-\nu)}{2\varepsilon} \right)^{1-\nu},
\tilde{L} = \left( L_{xy}\left(\frac{2L_{xy}}{\mu_y}\right)^{\frac{\nu}{2-\nu}}+L_{xx}D^{\frac{\nu-\nu^2}{2-\nu}} \right)$, \\$L_{xx},L_{xy}, L_{yy}>0$, $\nu$ is the H\"older exponent of $\nabla f$, $D$ is a diameter of the domain of $f(x,\cdot)$, $R$ denotes the distance between the initial point of the algorithm $x^0$ and the exact solution of the problem. Remind that $(\tilde{x},\tilde{y})$ is called an $\varepsilon$--solution of the saddle point problem, 
if 
\begin{equation}\label{eps_solut_saddle}
    \max\limits_y f(\tilde{x},y) - \min\limits_x f(x,\tilde{y})\leq \varepsilon.
\end{equation}

Let us note, that due to the strong convexity-concavity of the considered saddle point problem, the convergence in argument takes place with the similar asymptotic behavior. Thus, we can consider the introduced definition of the $\varepsilon$-solution \eqref{eps_solut_saddle} for the strongly convex-concave saddle point problem \eqref{main}.

The second part of the article continues with the study of accelerated methods and is devoted to the numerical experiments for some recently proposed \cite{UMP} universal algorithms. In particular, we consider the restarted version of the algorithm and apply it to the analogue of the covering circle problem with non-smooth functional constraints. We show, that the method can work faster, than $O(\frac{1}{\varepsilon}).$

Finally, the third part of the article explores Minty variational inequalities with the Relatively bounded operator. Recently Y. Nesterov proposed \cite{Nesterov_Relative_Smoothness} a generalization of the Lipschitz condition, which consists in abandoning the classical boundedness of the norm of the objective`s gradient $\nabla f$ in favor of a more complex structure. This structure allows one to take into account the peculiarities of the domain of the optimization problem, which can be effectively used in solving support vector machine (SVM) problem and intersection of $n$ ellipsoids problem \cite{LuRel,Titov_Rel}.

Remind the formulation of the Minty variational inequality. For a given operator $g(x):X\rightarrow\mathbb{R}$, where $X$ is a closed convex subset of some finite-dimensional vector space, we need to find a vector $x_*\in X$, such that
\begin{equation}\label{VI}
  \langle g(x),x_*-x\rangle\leq 0 \quad \forall x\in X.
\end{equation}

We consider the variational inequality problem under the assumption of Relative boundedness of the operator $g$, which is the modification of the aforementioned Relative Lipschitz condition for functionals. We introduce the modification of the Mirror Descent method to solve such variational inequality problems. The proposed method guarantees an $(\varepsilon+\sigma)$--solution of the problem after no more than $$\frac{2RM^2}{\varepsilon^2}$$ iterations, where $M$ is the Relative boundedness constant, which depends on the characteristics of the operator $g$ and $R$ can be understood as the distance between the initial point of the algorithm and the exact solution of the variational inequality in some generalized sense. The constant $\sigma$ reflects certain features of the monotonicity of the operator $g$. An $(\varepsilon+\sigma)$--solution of the variational inequality is understood as the point $\Tilde{x},$ such that
\begin{equation}\label{eps_VI}
    \max_{x\in X}\langle g(x),\Tilde{x}-x \rangle \leq \varepsilon+\sigma.
\end{equation}

The paper consists of the introduction and three main sections.
In Sect. \ref{section_Saddle_Point_Problem}, we consider the strongly convex-concave saddle point problem in a non-smooth setting.
Sect. \ref{sect_comparision}, is devoted to some numerical experiments concerning the methods, recently proposed in \cite{UMP} and the analysis of their asymptotics in comparison with the method proposed in Sect. \ref{section_Saddle_Point_Problem}.
In Sect. \ref{sect_mirror_VI}, we consider the Minty variational inequality problem with Relatively Bounded operator.

To sum it up, the contributions of the paper can be formulated as follows:

\begin{itemize}
	\item We consider the strongly convex-concave saddle point problem and its modified functional form. We prove, that the functional form admits an inexact $(\delta,L,\mu_x)$--model and apply the Fast Gradient Method to achieve the $\varepsilon$--solution of the considered problem. Moreover, we show, that $\delta = \mathcal{O}(\varepsilon).$ 
	\item We show, that the total number of iterations of the proposed method does not exceed
$$
    \mathcal{O}\left(\sqrt{\frac{L}{\mu_x}}\cdot\sqrt{\frac{L_{yy}}{\mu_y}}\cdot \log\frac{2L_{yy}R^2}{\varepsilon}\cdot\log{\frac{2LD^2}{\varepsilon}}\right),
$$
where for $0 \leq \nu < 1$, $L = \tilde{L}\left( \frac{\tilde{L} (1-\nu)}{2\varepsilon} \right)^{1-\nu},
\tilde{L} = \left( L_{xy}\left(\frac{2L_{xy}}{\mu_y}\right)^{\frac{\nu}{2-\nu}}+L_{xx}D^{\frac{\nu-\nu^2}{2-\nu}} \right)$.

    \item We introduce the modification of the Mirror Descent method to solve Minty variational inequalities with Relatively Bounded and $\sigma$--monotone operators.
	
    \item We show, that the proposed method can be applied to obtain an\\ $(\varepsilon+\sigma)$--solution after no more than $\frac{2RM^2}{\varepsilon^2}$ iterations.

\end{itemize}

\section{Accelerated Method for Saddle Point Problems with Generalized Smoothness Condition} \label{section_Saddle_Point_Problem}

Let $Q_x \subset \mathbb{R}^n$ and $Q_y \subset \mathbb{R}^m$ be nonempty, convex, compact sets and there exist $D >0$ and $R>0$, such that
$$
    \|x_1-x_2\|_2\leq D \quad  \forall x_1, \,x_2\in Q_x,
$$
$$
    \|y_1-y_2\|_2\leq R \quad  \forall y_1, \, y_2\in Q_y,
$$

Let  $f: Q_x\times Q_y\rightarrow\mathbb{R}$ be a  $\mu_x$-strongly convex function for fixed $y \in Q_y$ and $\mu_y$-strongly concave for fixed $x\in Q_x$. Remind, that differentiable function $h(x):Q_x\rightarrow\mathbb{R}$ is called $\mu$-strongly convex, if
$$\langle \nabla h(x_1)-\nabla h(x_2), x_1 - x_2\rangle\geq \mu\|x_1-x_2\|^2_2 \quad \forall x_1,x_2 \in Q_x. $$
Consider $(\mu_x,\mu_y)$-strongly convex--concave saddle point problem \eqref{main} under assumptions, that one of the partial gradients of $f$ satisfies the Lipschitz condition, while other three gradients satisfy the H\"older condition.
More formally, for any $x,x' \in Q_x, y,y'\in Q_y$ and for some $\nu \in [0,1]$, the following inequalities hold:
\begin{equation} \label{L_xx}
\| \nabla_xf(x,y) - \nabla_xf(x',y)\|_2 \leq L_{xx}\|x-x'\|_2^{\nu},
\end{equation} 
\begin{equation} \label{L_xy}
\| \nabla_xf(x,y) - \nabla_xf(x,y')\|_2 \leq L_{xy}\|y-y'\|_2^{\nu},
\end{equation} 
\begin{equation} \label{L_yx}
\| \nabla_yf(x,y) - \nabla_yf(x',y)\|_2 \leq L_{xy}\|x-x'\|_2^{\nu},
\end{equation} 
\begin{equation} \label{L_yy}
\| \nabla_yf(x,y) - \nabla_yf(x,y')\|_2 \leq L_{yy}\|y-y'\|_2.
\end{equation} 

Define the following function:
\begin{equation}\label{maximization_problem}
g(x) = \max\limits_{y\in Q_y} f(x,y), \quad x \in Q_x.
\end{equation}
It is obvious, that the considered saddle point problem \eqref{main} can be rewritten in the following more simple way:
\begin{equation}\label{modif}
g(x) = \max\limits_{y\in Q_y} f(x,y) \rightarrow \min\limits_{x\in Q_x}.
\end{equation}


Since $f(x,\cdot)$ is $\mu_y$-strongly concave on $Q_y$, the maximization problem \eqref{maximization_problem} has the unique solution
$$
    y^*(x) = \arg\max\limits_{y\in Q_y} f(x,y) \quad \forall x\in Q_x,
$$
so $g(x) = f(x,y^*(x)).$
Moreover, 
\begin{equation} \label{str_conc}
\|y^*(x_1)-y^*(x_2)\|_2^2\leq \frac{2}{\mu_y}\bigg(f\Big(x_1,y^*(x_1)\Big)-f\Big(x_1,y^*(x_2)\Big)\bigg)\;\; \forall x_1, x_2 \in Q_x. 
\end{equation} 

\begin{lemma}\label{lemm1}
Consider the problem \ref{main} under assumptions \eqref{L_xx}--\eqref{L_yy}. Define the function $g( x): Q_x \to \mathbb{R}$ according to \eqref{maximization_problem}. Then $g(x)$ has the H\"{o}lder continuous gradient with H\"{o}lder constant $\Biggl( L_{xy}\left(\frac{2L_{xy}}{\mu_y}\right)^{\frac{\nu}{2-\nu}}+L_{xx}D^{\frac{\nu-\nu^2}{2-\nu}} \Biggl)$ and H\"{o}lder exponent $\frac{\nu}{2-\nu}.$

\end{lemma}
\begin{proof}
Similarly to \cite{ST}, for any $x_1, x_2 \in Q_x$, let us estimate the following difference:
$$
    \bigg(f\Big(x_1,y^*(x_1)\Big)-f\Big(x_1,y^*(x_2)\Big)\bigg) - \bigg(f\Big(x_2,y^*(x_1)\Big) - f\Big(x_2,y^*(x_2)\Big)\bigg) \leq
$$
$$
    \leq \Big\|\nabla_xf\Big(x_1+t(x_2-x_1),y^*(x_1)\Big) - \nabla_x f\Big(x_1+t(x_2-x_1),y^*(x_2)\Big) \Big\|_2 \cdot \|x_2-x_1\|_2 \leq
$$
$$
    \leq L_{xy}\|y^*(x_1)-y^*(x_2)\|_2^{\nu}\cdot\|x_2 - x_1\|_2, \quad t\in[0,1].
$$

Using \eqref{str_conc}, one can get
$$
    \|y^*(x_1) - y^*(x_2)\|_2 \leq \left(\frac{2L_{xy}}{\mu_y}\right)^{\frac{1}{2-\nu}}\|x_2-x_1\|_2^{\frac{1}{2-\nu}},
$$
which means, that $y^*(x)$ satisfies the H\"{o}lder condition on $Q_x$ with H\"{o}lder constant $\left(\frac{2L_{xy}}{\mu_y}\right)^{\frac{1}{2-\nu}}$ and H\"{o}lder exponent $\frac{1}{2-\nu} \in [\frac{1}{2},1]$.

Further, 
$$
\| \nabla g(x_1) - \nabla g(x_2)\|_2 = \|\nabla_xf(x_1,y^*(x_1)) - \nabla_x f(x_2,y^*(x_2))\|_2 \leq$$
$$
\leq \|\nabla_x f(x_1,y^*(x_1)) - \nabla_x f(x_1, y^*(x_2))\|_2 + \|\nabla_x f(x_1,y^*(x_2))-\nabla_x f(x_2,y^*(x_2)) \|_2 \leq
$$
$$
\leq L_{xy}\|y^*(x_1)-y^*(x_2)\|_2^{\nu} + L_{xx}\|x_2-x_1\|_2^{\nu} = 
$$
$$ 
\leq L_{xy}\left(\frac{2L_{xy}}{\mu_y}\right)^{\frac{\nu}{2-\nu}}\|x_2-x_1\|_2^{\frac{\nu}{2-\nu}} + L_{xx}\|x_2-x_1\|_2^{\nu} = 
$$
$$ 
= L_{xy}\left(\frac{2L_{xy}}{\mu_y}\right)^{\frac{\nu}{2-\nu}}\|x_2-x_1\|_2^{\frac{\nu}{2-\nu}} + L_{xx}\|x_2-x_1\|_2^{\frac{\nu}{2-\nu}}\cdot \|x_2-x_1\|_2^{\frac{\nu-\nu^2}{2-\nu}}.
$$

Since $Q_x$ is bounded, we have
\begin{equation}\label{Hol_g}
\| \nabla g(x_1) - \nabla g(x_2)\|_2 \leq \Biggl( L_{xy}\left(\frac{2L_{xy}}{\mu_y}\right)^{\frac{\nu}{2-\nu}}+L_{xx}D^{\frac{\nu-\nu^2}{2-\nu}} \Biggl)\|x_2-x_1\|^{\frac{\nu}{2-\nu}}_2,
\end{equation}
which means, that $g(x)$ has the H\"{o}lder continuous gradient.
\end{proof}

\begin{definition}
A function $h(x):Q_x \to \mathbb{R}$ admits $(\delta,L,\mu)$-model, if, for any \\$x_1,x_2 \in Q_x$, the following inequalities hold:
\begin{dmath}\label{model}
\frac{\mu}{2}\|x_2-x_1\|_2^2 + \langle \nabla h(x_1),x_2-x_1 \rangle + h(x_1) - \delta \leq h(x_2) \leq h(x_1) + \langle \nabla h(x_1),x_2-x_1 \rangle + \frac{L}{2}\|x_2-x_1\|_2^2 + \delta
\end{dmath}
\end{definition}

\begin{remark}\cite{Gasnikov} 
Note, that if a function $h(x)$ has the  H\"{o}lder-continuous gradient with H\"{o}lder constant $\tilde{L}_{\tilde{\nu}}$ and H\"{o}lder exponent $\tilde{\nu}$, then $h(x)$ admits $(\delta,L,\mu)$-model. More precisely, inequalities \eqref{model} hold with $L = \tilde{L}_{\tilde{\nu}}\left(\frac{\tilde{L}_{\tilde{\nu}}}{2\delta}\frac{1-\tilde{\nu}}{1+\tilde{\nu}}  \right)^{\frac{1-\tilde{\nu}}{1+\tilde{\nu}}}.$
\end{remark}

\begin{remark}\label{rem2}
According to Lemma \ref{lemm1}, $g(x)$ has the H\"{o}lder continuous gradient, so $g(x)$ admits $(\delta_0,L,\mu_x)$-model ($\delta$ was replaced by $\delta_0$ to simplify notation) with 
$$
L = \tilde{L}\left( \frac{\tilde{L} (1-\nu)}{2\varepsilon} \right)^{1-\nu},
$$
where
$\tilde{L} = \left( L_{xy}\left(\frac{2L_{xy}}{\mu_y}\right)^{\frac{\nu}{2-\nu}}+L_{xx}D^{\frac{\nu-\nu^2}{2-\nu}} \right).$
\end{remark}

Let us now assume, that instead of the ordinary gradient $\nabla g(x)$ we are given an inexact one $\widetilde{\nabla} g(x)\coloneqq \nabla_x f(x,\widetilde{y})$, such that 
\begin{equation}\label{acc_y}
\|y_1-\widetilde{y}\|_2\leq \widetilde{\Delta},
\end{equation}
where $f(x,y_1) = \max\limits_{y\in Q_y}f(x,y), \widetilde{\Delta} = const. > 0.$

\begin{theorem}
Consider the strongly convex-concave saddle point problem \eqref{main} under assumptions \eqref{L_xx}-\eqref{L_yy}. Define the function $g(x)$ according to \eqref{modif}. Then $g(x)$ admits an inexact $(\delta,L,\mu_x)$-model with $\delta = (D\Delta + \delta_{0})$ and $L$ defined in Remark \ref{rem2}. Applying $k$ steps of the Fast Gradient Method to the "outer" problem \eqref{modif} and solving the "inner" problem \eqref{acc_y} in linear time, we obtain an $\varepsilon$-solution \eqref{eps_solut_saddle} to the problem \eqref{main},
where $\delta = \mathcal{O}(\varepsilon)$. The total number of iterations does not exceed 
$$
    \mathcal{O}\left(\sqrt{\frac{L}{\mu_x}}\cdot\sqrt{\frac{L_{yy}}{\mu_y}}\cdot \log\frac{2L_{yy}R^2}{\varepsilon}\cdot\log{\frac{2LD^2}{\varepsilon}}\right),
$$
where $L = \tilde{L}\left( \frac{\tilde{L} (1-\nu)}{2\varepsilon} \right)^{1-\nu},
\tilde{L} = \left( L_{xy}\left(\frac{2L_{xy}}{\mu_y}\right)^{\frac{\nu}{2-\nu}}+L_{xx}D^{\frac{\nu-\nu^2}{2-\nu}} \right)$. 
\end{theorem}

\begin{proof}
Since the problem $\max\limits_{y\in Q_y}f(x_1,y)$ is smooth (\eqref{L_yy} holds) and $\mu_y$-strongly concave, one can achieve any arbitrary accuracy in \eqref{acc_y}, furthermore, in linear time \cite{Gasnikov}. More formally, for any $\widetilde{\Delta} >0$, in particular, $\widetilde{\Delta} = \left(\frac{\Delta}{L_{xy}}\right)^{\frac{1}{\nu}}$, there exists  $\widetilde{y}$, such that
$$
  \|\widetilde{y}-y_1\|_2 \leq \left(\frac{\Delta}{L_{xy}}\right)^{\frac{1}{\nu}}.
$$

Then
\begin{equation}
\|\widetilde{\nabla} g(x_1) - \nabla g(x_1)\|_{2} = \|\nabla_x f(x_1,\widetilde{y}) - \nabla_x f(x_1,y_1)\|_{2}\leq 
L_{xy}\|\widetilde{y} - y_1\|^{\nu}_2 \leq \Delta.
\end{equation}

Thus, taking into account that $$\langle \nabla g(x_1), x_2-x_1 \rangle = \langle \nabla g(x_1) - \widetilde{\nabla} g(x_1), x_2-x_1 \rangle + \langle \widetilde{\nabla} g(x_1), x_2-x_1 \rangle,$$ we get:
 
\begin{dmath}\label{mod_res}
\frac{\mu_x}{2}\|x_2-x_1\|_2^2 + \langle \widetilde{\nabla} g(x_1),x_2-x_1 \rangle + g(x_1) - D\Delta-\delta_0\leq\\ \leq g(x_2) \leq g(x_1) + \langle \widetilde{\nabla} g(x_1),x_2-x_1 \rangle + \frac{L}{2}\|x_2-x_1\|_2^2 + D\Delta + \delta_{0}.
\end{dmath}

The inequality \eqref{mod_res} means, that $g(x)$ admits the inexact $(\delta,L,\mu_x)$-model with $\delta = (D\Delta + \delta_{0})$ and $L$ defined in Remark \ref{rem2}.

It is well known \cite{Gasnikov}, that using Fast Gradient Method for the described construction, after $k$ iterations, one can obtain the following accuracy of the solution:
$$g(x^k) - g(x^*) \leq LR^2\exp{\left(-\frac{k}{2}\sqrt{\frac{\mu_x}{L}}\right)} + \delta\left(1+\sqrt{\frac{L}{\mu_x}}\right).$$

As noted above, it is possible to achieve an arbitrarily small value of $\widetilde{\Delta}$, more precisely, such $\widetilde{\Delta}$, that

$$
DL_{xy}\widetilde{\Delta}^{\nu} +\delta_0= D\Delta+\delta_0 = \delta \leq \frac{\varepsilon}{2\left(1+ \sqrt{\frac{L}{\mu_x}}\right)}.
$$
In other words, we can obtain the appropriate $\widetilde{\Delta}$, such that $\delta$ will be comparable to $\varepsilon$. For example, we can put
$$
\delta_0 \leqslant \frac{\varepsilon}{4\left(1+ \sqrt{\frac{L}{\mu_x}}\right)}, \quad \Delta \leqslant \frac{\varepsilon}{4D\left(1+ \sqrt{\frac{L}{\mu_x}}\right)}.
$$

Thus, after no more than $k =2\sqrt{\frac{L}{\mu_x}} \log{\frac{2LR^2}{\varepsilon}}$ iterations we obtain an\\ $\varepsilon$-solution \eqref{eps_solut_saddle} of the considered saddle point problem \eqref{main}.

The total number of iterations (including achieving the appropriate precision of $\widetilde{\Delta}$) is expressed as follows:
$$
    \mathcal{O}\left(\sqrt{\frac{L}{\mu_x}}\cdot\sqrt{\frac{L_{yy}}{\mu_y}}\cdot \log\frac{2L_{yy}R^2}{\varepsilon}\cdot\log{\frac{2LD^2}{\varepsilon}}\right),
$$
where $L = \tilde{L}\left( \frac{\tilde{L} (1-\nu)}{2\varepsilon} \right)^{1-\nu},
\tilde{L} = \left( L_{xy}\left(\frac{2L_{xy}}{\mu_y}\right)^{\frac{\nu}{2-\nu}}+L_{xx}D^{\frac{\nu-\nu^2}{2-\nu}} \right).$
\end{proof}



\section{Comparison of Theoretical Results for Accelerated Method and the Universal Proximal Method for Saddle-Point Problems with Generalized Smoothness}\label{sect_comparision}

Let us investigate the effectiveness of the proposed method for solving strongly convex--concave saddle point problems in comparison with the Universal Algorithm, which was recently proposed in \cite{UMP} to solve variational inequalities. We have to start with the problem statement and define all basics concerning variational inequalities and Proximal Setup. Note, that in Sect. \ref{section_Saddle_Point_Problem} we used exclusively the Euclidean norm, while results of Sect. \ref{sect_comparision} and Sect. \ref{sect_mirror_VI} hold for an arbitrary norm.

Let $E$ be some finite-dimensional vector space, $E^*$ be its dual. Let us choose some norm $\|\cdot\|$ on $E$. Define the dual norm $\|\cdot\|_*$ as follows:
$$\|\phi\|_{*} = \max\limits_{\|x\|\leq 1} \{ \langle \phi, x \rangle\},$$
where $\langle \phi, x \rangle$ denotes the value of the linear function $\phi \in E^*$ at the point $x\in E$.

Let $X\subset E$ be a closed convex set 
and $g(x):X\rightarrow E^*$ be a monotone operator, i.e.
\begin{equation}
    \left\langle g(x) -g(y), x- y \right\rangle \geq 0 \quad \forall x,y \in X.
\end{equation}



We also need to choose a so-called prox-function $d(x)$, which is continuously differentiable and convex on $X$, and the corresponding Bregman divergence, which can be understood as an analogue of the distance, defined as follows: 
\begin{equation}\label{Bregmann_def}
    V(y,x) = V_d(y,x) = d(y) - d(x) - \langle \nabla d(x), y - x \rangle \quad \forall x, y \in X.
\end{equation}

In \cite{UMP}, the authors proposed the Universal Proximal method ({\tt UMP}, this method is listed as Algorithm 1, below) for solving the problem \eqref{VI} with inexactly given operator. In details, suppose that there exist some $\delta >0, L(\delta)>0$, such that, for any points $x,y,z\in X$, we are able to calculate $\Tilde{g}(x,\delta), \Tilde {g}(y,\delta) \in E^*$, satisfying
\begin{equation}
	\label{eq:g_or_def}
	\langle \Tilde{g}(y,\delta) - \tilde{g}(x,\delta), y - z \rangle \leq \frac{L(\delta)}{2}\left(\|y-x\|^2 + \|y-z\|^2\right) + \delta \quad \forall z \in Q.
\end{equation}

In addition, there was considered the possibility of using the restarted version of the method ({\tt Restarted UMP}, see Algorithm 2 below) in the case of $\mu$--strongly monotone operator $g$:
$$
\langle g(x) - g(y), x-y\rangle \geq \mu \|x-y\|^2 \quad \forall x, y \in X.
$$
We additionally assume that $\argmin_{x \in X}d(x) = 0$ and  $d(\cdot)$ is bounded on the unit ball in the chosen norm $\|\cdot\|$, more precisely
$$
d(x) \leq \frac{\Omega}{2} \quad \forall x \in X: \|x\|\leq 1,
$$
where $\Omega$ is a known constant.

\begin{algorithm}[h!]
\caption{Universal Mirror Prox (\tt{UMP})}
\label{Alg:UMP}
\begin{algorithmic}[1]
   \REQUIRE $\varepsilon > 0$, $\delta >0$, $x_0 \in X$, initial guess $L_0 >0$, prox-setup: $d(x)$, $V(x,z)$.
   \STATE Set $k=0$, $z_0 = \arg \min_{u \in Q} d(u)$.
   \FOR{$k=0, 1,...$}
			\STATE Set $M_k=L_k/2$.
			\STATE Set $\delta = \frac{\varepsilon}{2}$.
			\REPEAT
				\STATE Set $M_k=2M_k$.
				\STATE Calculate $\tilde{g}(z_k,\delta)$ and
				\begin{equation}
				w_k={\mathop {\arg \min }\limits_{x\in Q}} \left\{\langle \tilde{g}(z_k,\delta),x \rangle + M_kV(x,z_k) 		 \right\}.
				\label{eq:UMPwStep}
				\end{equation}
				\STATE Calculate $\tilde{g}(w_k,\delta)$ and
				\begin{equation}
				z_{k+1}={\mathop {\arg \min }\limits_{x\in Q}} \left\{\langle \tilde{g}(w_k,\delta),x \rangle + M_kV(x,z_k) 		 \right\}.
				\label{eq:UMPzStep}
				\end{equation}
			\UNTIL{
			\begin{equation}
			\langle \tilde{g}(w_k,\delta) - \tilde{g}(z_k,\delta), w_k - z_{k+1} \rangle \leq \frac{M_k}{2}\left(\|w_k-z_k\|^2 + \|w_k-z_{k+1}\|^2\right) + \frac{\varepsilon}{2}+\delta.
			\label{eq:UMPCheck}
			\end{equation}}
			\STATE Set $L_{k+1}=M_k/2$, $k=k+1$.
   \ENDFOR
		\ENSURE $z_k = \frac{1}{\sum_{i=0}^{k-1}M_i^{-1}}\sum_{i=0}^{k-1}M_i^{-1}w_i$.
\end{algorithmic}
\end{algorithm}

\begin{algorithm}[htp]
	\caption{Restarted Universal Mirror Prox ({\tt Restarted UMP}).}
	\label{Alg:RUMP}
	\begin{algorithmic}[1]
		\REQUIRE $\varepsilon > 0$, $\mu >0$, $\Omega$ : $d(x) \leq \frac{\Omega}{2} \ \forall x\in Q: \|x\| \leq 1$; $x_0,\; R_0 \ : \|x_0-x_*\|^2 \leq R_0^2.$
		\STATE Set $p=0,d_0(x)=R_0^2d\left(\frac{x-x_0}{R_0}\right)$.
		\REPEAT
		\STATE Set $x_{p+1}$ as the output of {\tt UMP} for monotone case with prox-function $d_{p}(\cdot)$ and stopping criterion $\sum_{i=0}^{k-1}M_i^{-1}\geq \frac{\Omega}{\mu}$.
		\STATE Set $R_{p+1}^2 = R_0^2 \cdot 2^{-(p+1)} + 2(1-2^{-(p+1)})\frac{\varepsilon}{4\mu}$.
		\STATE Set $d_{p+1}(x) \leftarrow R_{p+1}^2d\left(\frac{x-x_{p+1}}{R_{p+1}}\right)$.
		\STATE Set $p=p+1$.
		\UNTIL			
		$p > \log_2\left(\frac{2R_0^2}{\varepsilon}\right).$	
		\ENSURE $x_p$.
	\end{algorithmic}
\end{algorithm}

\begin{remark}\label{saddle_VI}(Connection between saddle point problem and VI). Any saddle point problem 
$$\min\limits_{x\in Q_x}\max\limits_{y\in Q_y}f(x,y)$$
can be reduced to a variational inequality problem by considering the following operator:
\begin{equation}\label{SP_VI}
g(z)=
\begin{pmatrix}
\nabla_x f(x,y)\\
-\nabla_y f(x,y)
\end{pmatrix}, \; z=(x,y) \in Q:=Q_x \times Q_y.
\end{equation}
\end{remark}


\begin{remark}

{\tt Restarted UMP} method returns a point $x_p$ such that $$\|x_p-x_*\|^2 \leq \varepsilon + \frac{2\delta}{\mu}.$$ 
Moreover, the number of calls to Algorithm 1 ({\tt UMP}) while running Algorithm 2 ({\tt Restarted UMP}) does not exceed 
$$\inf_{\nu \in [0,1]} \left\lceil  \left(\frac{L_{\nu}}{\mu}\right)^{\frac{2}{1+\nu}}\cdot \frac{2^{\frac{2}{1+\nu}}\Omega }{\varepsilon^{\frac{1-\nu}{1+\nu}}} \cdot \log_2 \frac{2 R_0^2}{\varepsilon}\right\rceil.$$
Note, that due to the adaptivity of {\tt UMP}, in practice the number of calls may decrease (it will be shown experimentally later).
\end{remark}

\begin{remark}
Note, that for $\nu = 0$  the convergence rate of the {\tt Restarted UMP} and the accelerated method, introduced in Sect. \ref{section_Saddle_Point_Problem}, coincides, while for $\nu>0$ the asymptotic of the proposed accelerated method is better.
\end{remark}

\subsection{Numerical Experiments for Restarted Universal Mirror Prox Method}
The problem of constrained minimization of convex functionals arises and attracts widespread interest in many areas of modern large-scale optimization and its applications \cite{Robust_Truss,Shpirko_2014}.


In this subsection, in order to demonstrate the performance of the {\tt Restarted UMP},  we consider an example of the Lagrange saddle point problem induced by a problem with geometrical nature, namely, an analogue of the well-known {\it smallest covering ball problem} with non-smooth functional constraints. This example is equivalent to the following non-smooth convex optimization problem with functional constraints 
\begin{equation}\label{problem_smallest}
	\min_{x \in Q} \left\{f(x):= \max_{1\leq k \leq N} \|x - A_k\|_2^2; \; \varphi_p(x) \leq 0, \; p=1,...,m \right\},
\end{equation}
where $A_k \in \mathbb{R}^n, k=1,...,N$ are given points and $Q$ is a convex compact set. Functional constraints $\varphi_p$,  for $p=1,...,m$, have the following form:
\begin{equation}\label{linear_constraints}
	\varphi_p(x):= \sum_{i=1}^n \alpha_{pi}x_i^2 -5, \; p = 1, ..., m.
\end{equation}

For solving such problems there are various first-order methods, which guarantee achieving an acceptable precision $\varepsilon$ by function with complexity $\mathcal{O}\left(\varepsilon^{-1}\right)$. We present the results of some numerical experiments, which demonstrate the effectiveness of the {\tt Restarted UMP}. Remind, that we analyze how the restart technique can be used to improve the convergence rate of the {\tt UMP}, and show that in practice it works with a convergence rate smaller than $\mathcal{O}\left(\varepsilon^{-1}\right)$, for some different randomly generated data associated with functional constraints \eqref{linear_constraints}. 




The corresponding Lagrange saddle point problem of the problem \eqref{problem_smallest} is defined as follows
$$
    \min_{x \in Q} \max_{\overrightarrow{\lambda} = (\lambda_1,\lambda_2,\ldots,\lambda_m)^T \in \mathbb{R}^m_+} L(x,\lambda):=f(x)+\sum\limits_{p=1}^m\lambda_p\varphi_p(x) -\frac{1}{2}\sum\limits_{p=1}^{m}\lambda_p^2.
$$ 
This problem is satisfied to \eqref{L_xx} -- \eqref{L_yy} for $\nu = 0$ and equivalent to the variational inequality with the monotone bounded operator
$$
    G(x,\lambda)=
    \begin{pmatrix}
    \nabla f(x)+\sum\limits_{p=1}^m\lambda_p\nabla\varphi_p(x), \\
    (-\varphi_1(x) +\lambda_1,-\varphi_2(x)+\lambda_2,\ldots,-\varphi_m(x) +\lambda_m)^T
    \end{pmatrix},
$$
where $ \nabla f$ and $\nabla\varphi_p$ are subgradients of $f$ and $\varphi_p$. For simplicity, let us assume that there exists (potentially very large) bound for the optimal Lagrange multiplier $\overrightarrow{\lambda}^*$. Thus, we are able to compactify the feasible set for the pair $(x,\overrightarrow{\lambda})$ to be an Euclidean ball of some radius.

To demonstrate the independence on the choice of experimental data, the coefficients $\alpha_{pi}$ in \eqref{linear_constraints} are drawn randomly from four different distributions.

\begin{itemize}
    \item \textbf{Case 1:} the standard exponential distribution.
	
	\item \textbf{Case 2:} the Gumbel distribution with mode and scale equal to zero and 1, respectively.
	
	\item \textbf{Case 3:} the inverse Gaussian distribution with mean and scale equal to 1 and 2, respectively. 
	
	 \item \textbf{Case 4:} from the discrete uniform distribution in the half open interval $[1,6)$.
\end{itemize}

We run {\tt Restarted UMP} for different values of  $n$ and  $m$ with standard Euclidean prox-structure and the starting point $(x^0, \overrightarrow{\lambda}^0) = \frac{1}{\sqrt{m+n}} \textbf{1} \in \mathbb{R}^{n+m}$, where $\textbf{1} $ is the vector of all ones. Points $A_k$, $k=1,...,N$, are chosen randomly from the uniform distribution over $[0,1)$. For each value of the parameters the random data was drawn 5 times and the results were averaged. 
The results of the work of {\tt Restarted UMP} are presented in Table \ref{table_results1}. 
These results demonstrate the number of iterations produced by {\tt Restarted UMP} to reach the $\varepsilon$-solution of the problem \eqref{problem_smallest} with \eqref{linear_constraints}, the running time of the algorithm in seconds, qualities of the solution with respect to the objective function $f$ ($f^{\text{best}} := f(x_{\text{out}})$) and the functional constraints $g$ ($g^{\text{out}} := g(x_{\text{out}})$), where $x_{\text{out}}$ denotes the output of the compared algorithms, with different values of $\varepsilon \in \{ 1/2^i, i=1,2,3,4,5,6\}$. 

All experiments were implemented in Python 3.4, on a computer fitted with Intel(R) Core(TM) i7-8550U CPU @ 1.80GHz, 1992 Mhz, 4 Core(s), 8 Logical Processor(s). RAM of the computer is 8 GB.

\begin{table}[htp]
	\centering
	\caption{The results of {\tt Restarted UMP}, for the problem \eqref{problem_smallest} with constraints \eqref{linear_constraints}  for \textbf{Cases 1 -- 4}.}
	\begin{tabular}{|c|c|c|c|c||c|c|c|c|}
		\hline
		\multirow{2}{*}{} & \multicolumn{4}{c||}{{\textbf{Case 1:} $n = 1000, m = 50, N = 10$}} & \multicolumn{4}{c|}{\textbf{Case 2:} $n = 1000, m = 50, N = 10$} \\ \cline{2-9}  
		$\frac{1}{\varepsilon}$	& Iter. &  Time (sec.) &$f^{\text{best}}$ & $g^{\text{out}}$ & Iter. & Time (sec.)& $f^{\text{best}}$&$g^{\text{out}}$\\ \hline
		2 &9  &0.392 &324.066325&-4.349744 &12 &0.517 &   324.889649&-3.965223     \\ 
		4 &12 &0.519 &324.066312  &-4.349732 &16 &0.630 &324.889634 & -3.965213   \\ 
		8 &15 &0.599 &324.066305&-4.349692 &20 &0.807 &324.889621 & -3.965197   \\ 
		16&18&0.972 &324.066295&-4.349680 & 24&0.977 &324.889598 & -3.965165  \\
		32&21  &0.984&324.066291&-4.349653 &28 &1.224 & 324.889540&-3.965223\\ 
		64&24  &1.357&324.066286  &-4.349598 &32 &1.317 &   324.889527&-3.965111 \\ \hline
		\multirow{2}{*}{} & \multicolumn{4}{c||}{{\textbf{Case 3:} $n = 500, m = 25, N = 10$}} & \multicolumn{4}{c|}{\textbf{Case 4:} $n = 500, m = 25, N = 10$} \\ \cline{2-9} 
		 \hline
		 2 &624  &20.321 &153.846214 &-4.398782 &832 &26.629 &158.210875 & -2.610276    \\ 
		4 &1319  &39.045 &153.842306 &-4.398106 &1702 &51.258 &158.201645 &-2.601072    \\ 
		8 &2158  &58.919  &153.830012  &-4.397387 &4144 &96.136 &158.190455 &-2.599223    \\ 
		16&4298  &117.383 &153.827731 &-4.397271 &6145 &174.713 &158.188211 &-2.598865   \\
		32&8523 &264.777 &153.826829 &-4.397226 &12081 &351.520 &158.187255 &-2.598713 \\ 
		64&17584  &554.480 &153.826382 &-4.397204 &30186 &768.861 &158.186744 &-2.598628 \\ \hline
		
	\end{tabular}
	\label{table_results1}
\end{table}

From the results in Table \ref{table_results1} we can see that the work of {\tt Restarted UMP}, does not only dependent on the dimension $n$ of the problem, the number of the constraints $\varphi_p,\, p=1, ..., m$ and the number of the points $A_k,\, k=1, ..., N$, but also depends on the shape of the data, that generated to the coefficients $\alpha_{pi}$ in \eqref{linear_constraints}. Also, these results demonstrate how the {\tt Restarted UMP}, due to its adaptivity to the level of smoothness of the problem, in practice works with a convergence rate smaller than $\mathcal{O}\left(\varepsilon^{-1}\right)$, for all different cases 1--4.




\section{Mirror Descent for Variational Inequalities with Relatively Bounded Operator}\label{sect_mirror_VI}

Let $g(x): X \to E^*$, be an operator, given on some convex compact set $X \subset E$. Recall that $g(x)$ is bounded on $X$, if there exists $M> 0$, such that
$$\| g(x)\|_*\leq M \quad \forall x \in X.$$
We can replace the classical concept of the boundedness of an operator by the so-called  Relative boundedness condition as following.

\begin{definition}(The Relative boundedness).
An operator $g(x):X \to E^*$ is Relatively bounded on $X$, if there exists $M> 0$, such that
\begin{equation}\label{Rel_Boud}
\langle g(x),y-x\rangle \leq M\sqrt{2V(y,x)} \quad \forall x,y \in X,
\end{equation}
where $V(y,x)$ is the Bregmann divergence, defined in \eqref{Bregmann_def}.
\end{definition}

Let us note the following special case of the definition.
\begin{remark}The Relative boundedness condition can be rewritten in the following way:
$$\| g(x)\|_*\leq \frac{M\sqrt{2V(y,x)}}{\|y-x\|} \ \ y\neq x.$$
\end{remark}

In addition to the Relative boundedness condition, suppose that the operator $g(x)$ is $\sigma$-monotone.

\begin{definition}($\sigma$-monotonicity). 
Let $\sigma>0$. The operator $g(x):X \to E^*$ is $\sigma$-monotone, if the following inequality holds:
\begin{equation}\label{d_monot}
\langle g(y)-g(x),y-x\rangle \geq -\sigma \quad \forall x, y \in X.
\end{equation}
\end{definition}
For example, we can consider $g = \nabla_{\sigma} f$ for $\sigma$-subgradient $\nabla_{\sigma} f(x)$ of the convex function $f$ at the point $x \in X$: $f(y) - f(x) \geqslant \langle  \nabla_{\sigma} f(x), y - x \rangle - \sigma $ for each $y \in X$ (see e.g. \cite{Polyak}, Chapter 5). Let us propose an analogue of the Mirror Descent algorithm for variational inequalities with Relatively bounded and $\sigma$-monotone operator. For any $x \in X$ and $p \in E^*$, we define the Mirror Descent step $Mirr_x(p)$ as follows:
$$Mirr_x(p) = \argmin\limits_{y\in X}\Big\{ \langle p,y \rangle +V(y,x)\Big\}.$$

\begin{algorithm}
\caption{Mirror Descent method for variational inequalities.}
\label{algggrest}
\begin{algorithmic}[1]
\REQUIRE $\varepsilon >0, \, M >0; x_0$ and $ R\ \text{such that}\ \max_{x \in X} V(x,x_0)\leq R^2.$
\STATE Set $h=\frac{\displaystyle\varepsilon}{\displaystyle M^2}.$
\STATE Initialization $k=0.$
\REPEAT
\STATE $x_{k+1}=Mirr_{x_k}\Big(hg(x_k)\Big).$
\STATE Set $k = k+1. $ 
\UNTIL {$k\geq N = \frac{\displaystyle 2RM^2}{\displaystyle\varepsilon^2}.$}
\ENSURE $\Tilde x = \frac{1}{N}\sum_{k=0}^{N-1}x_k.$
\end{algorithmic}
\end{algorithm}

The following theorem describes the effectiveness of the proposed Algorithm \ref{algggrest}.
\begin{theorem}
       Let  $g: X \to E^*$ be Relatively bounded and $\sigma$--monotone operator, i.e. \eqref{Rel_Boud} and \eqref{d_monot} hold. Then after no more than $$N = \frac{2RM^2}{\varepsilon^2}$$ iterations of Algorithm \ref{algggrest}, one can obtain an $(\varepsilon+\sigma)$--solution \eqref{eps_VI} of the problem \eqref{VI}, i.e.
      $$
    \max_{x\in X}\langle g(x),\Tilde{x}-x \rangle \leq \varepsilon+\sigma.
$$
\end{theorem}

\begin{proof}
According to the listing of the Algorithm \ref{algggrest}, the following inequality holds:
$$
h\langle g(x_k),x_k-x\rangle \leq \frac{h^2M^2}{2}+V(x,x_k)-V(x,x_{k+1}).
$$

Taking summation over $k = 0,1,\ldots, N-1$, we get
$$
    h\sum_{k=0}^{N-1}\left\langle g(x_k), x_k-x\right\rangle \leq \sum_{k=0}^{N-1}\frac{h^2M^2}{2}+V(x,x_0)-V(x,x_{N})\leq 
$$
$$
   \quad \leq \frac{N}{2}h^2M^2+V(x,x_0).
$$
Due to the $\sigma$-monotonicity of $g$, we have
$$\langle g(x_k),x_k-x\rangle \geq \langle g(x),x_k-x\rangle - \sigma.$$
Whence,
$$
    h\sum_{k=0}^{N-1}\left(\left\langle g(x_k),x_k-x \right\rangle +\sigma\right) \geq h\left\langle g(x),\sum_{k=0}^{N-1} (x_k-x) \right\rangle = h\left\langle g(x),N(\Tilde{x}-x) \right\rangle,
$$
where $\Tilde{x}=\frac{1}{N}\sum_{k=0}^{N-1}x_k.$
Since
$$
    N h\langle g(x),\Tilde{x}-x \rangle \leq \frac{N}{2}h^2M^2+V(x,x_0)+Nh\sigma,
$$
we get
$$
    \langle g(x),\Tilde{x}-x \rangle \leq \frac{M^2h}{2}+\frac{V(x,x_0)}{Nh}+\sigma.
$$

As $h=\frac{\varepsilon}{M^2}$, after the stopping criterion $N\geq \frac{2RM^2}{\varepsilon^2}$ is satisfied, we obtain an $(\varepsilon + \sigma)$-solution of the problem \eqref{VI}.
\end{proof}

\section{Conclusions}

In the first part of the article we considered strongly convex-concave saddle point problems; using the concept of $(\delta,L,\mu_x)$--model we applied the Fast Gradient Method to obtain an $\varepsilon$-solution of the problem. We proved, that the proposed method can generate an $\varepsilon$-solution of the saddle point problem after no more than 
$$
    \mathcal{O}\left(\sqrt{\frac{L}{\mu_x}}\cdot \sqrt{\frac{L_{yy}}{\mu_y}}\cdot \log\frac{2L_{yy}R^2}{\varepsilon}\cdot\log{\frac{2LD^2}{\varepsilon}}\right)\quad\text{iterations,}
$$
where for $0 \leq \nu < 1$ $L = \tilde{L}\left(\frac{\tilde{L}(1-\nu)}{2\varepsilon}\right)^{1-\nu},
\tilde{L} = \left( L_{xy}\left(\frac{2L_{xy}}{\mu_y}\right)^{\frac{\nu}{2-\nu}}+L_{xx}D^{\frac{\nu-\nu^2}{2-\nu}} \right).$\\
We conducted some numerical experiments for the Universal Proximal algorithm and its restarted version and analyzed their asymptotics in comparison with the proposed method. 

Also, we considered Minty variational inequalities with Relatively bounded and $\sigma$--monotone operator. We introduced the modification of the Mirror Descent method and proved, that obtaining an $(\varepsilon+\sigma)$-solution will take no more than 
 $\frac{2RM^2}{\varepsilon^2}$ iterations.

\bigskip

\textbf{Acknowledgment:} The authors are grateful to Mahmoud Karafallah, who drew the authors' attention to a typo in the estimation of the constant $L$ from the first version of the article.

\end{document}